\documentclass[11pt]{amsart}

\usepackage{amsmath, amssymb}
\usepackage{graphicx}
\usepackage[usenames]{xcolor}
\usepackage{float}

\usepackage[margin=1.5in]{geometry}
\definecolor{darkblue}{rgb}{0,0,0.7}
\definecolor{darkred}{rgb}{0.7,0,0}
\usepackage[colorlinks,linkcolor=darkred, citecolor=darkblue,
pagebackref=true,pdftex]{hyperref}

\newcommand\Defn[1]{\textit{\color{blue}#1}}    % Definitions in the text

\newcommand\Lap{\mathcal{L}}

\newcommand\R{\mathbb{R}}
\newcommand\Z{\mathbb{Z}}
\newcommand\N{\mathbb{N}}
\newcommand\A{\mathcal{A}}
\newcommand\tA{\tilde{\A}}
\newcommand\B{\mathcal{B}}
\newcommand\x{\mathbf{x}}
\renewcommand\th{\tilde{h}}
\renewcommand\k{\mathbb{K}}

\DeclareMathOperator*\supp{supp}
\DeclareMathOperator*\relint{relint}

\renewcommand\Lap{\mathcal{L}}
\newcommand\Laplat{\Lambda}
\newtheorem{thm}{Theorem}[section]
\newtheorem{prop}[thm]{Proposition}
\newtheorem{lemma}[thm]{Lemma}

\newtheorem{cor}[thm]{Corollary}

\newtheorem{rem}[thm]{Remark}
\newtheorem{example}[thm]{Example}

\begin{document}

\title[Laplacian ideals, arrangements, and resolutions]{Laplacian ideals,\\
arrangements, and resolutions}

\author{Anton Dochtermann}
\address{Department of Mathematics, University of Miami, USA}
\email{anton@math.miami.edu}

\author{Raman Sanyal}
\address{Fachbereich Mathematik und Informatik, %
Freie Universit\"at Berlin, %
Germany}
\email{sanyal@math.fu-berlin.de}


\date{\today}

\thanks{Raman Sanyal has been supported by the
European Research Council under the European Union's Seventh Framework
Programme (FP7/2007-2013) / ERC grant agreement n$^\mathrm{o}$ 247029.}

\begin{abstract}
    The Laplacian matrix of a graph $G$ describes the combinatorial dynamics
    of the Abelian Sandpile Model and the more general Riemann-Roch theory of
    $G$.  The lattice ideal associated to the lattice generated by the columns
    of the Laplacian provides an algebraic perspective on this recently
    (re)emerging field.  This binomial ideal $I_G$ has a distinguished
    monomial initial ideal $M_G$, characterized by the property that the
    standard monomials are in bijection with the $G$-parking functions of the
    graph $G$.  The ideal $M_G$ was also considered by Postnikov and Shapiro
    (2004) in the context of monotone monomial ideals.  We study resolutions
    of $M_G$ and show that a minimal free cellular resolution is supported on
    the bounded subcomplex of a section of the graphical arrangement of $G$.
    This generalizes constructions from Postnikov and Shapiro (for the case of
    the complete graph) and connects to work of Manjunath and Sturmfels, and
    of Perkinson et al.\ on the commutative algebra of Sandpiles. As a
    corollary we verify a conjecture of Perkinson et al.\ regarding the Betti
    numbers of $M_G$, and in the process provide a combinatorial
    characterization in terms of acyclic orientations.

\end{abstract}
\maketitle

\section{Introduction}\label{sec:intro}

Let $G = (V,E)$ be a undirected and connected graph with vertex set $V = [n+1]
= \{ 1,2,\dots, n+1\}$.  The \emph{Laplacian} $\Lap(G)$ of $G$ is a symmetric
$(n+1) \times (n+1)$ matrix that encodes the dynamics of the \emph{chip-firing
game} on the graph $G$.  More recently the Laplacian has been central to the
study of a discrete version of the Riemann-Roch theorem for graphs, where
chip-firing serves as the graph theoretic notion of linear equivalence of
divisors.

In this paper we are interested in a certain class of ideals arising from
$\Lap(G)$.  We fix a field $\k$ and consider the lattice ideal $I_G \subset
\k[x_1, \dots x_{n+1}]$ associated to $\Lap(G)$.  By definition $I_G$ is
generated by binomials of the form ${\bf x^u} - {\bf x^v}$ where $\bf{u},
\bf{v} \in \N^{n+1}$ and $\bf{u} - \bf{v}$ is in the lattice spanned by the
rows of $\Lap(G)$.  Following the lead of \cite{MS} we call this ideal the
\emph{toppling ideal} of the graph $G$; it was first introduced by Perkinson,
Perlman and Wilmes in \cite{PPW}.

After fixing the vertex $n+1$, the ideal $I_G$ has a distinguished initial
monomial ideal $M_G$ with the property that the standard monomials of $M_G$
are in bijection with the so-called \emph{$G$-parking functions}.  This
monomial ideal was first studied by Postnikov and Shapiro~\cite{PS} in the
context of monotone monomial ideals and their deformations, and can be defined
by an explicit combinatorial rule (see below).  As is illustrated in
\cite{MS}, the ideal $M_G$ has interesting connections to the Riemann-Roch
theory of $G$.

In each of the papers \cite{MS}, \cite{PPW}, \cite{PS} various free
resolutions of the ideals $I_G$ and $M_G$ are considered.  In the case that $G
= K_{n+1}$ is a complete graph on $n+1$ vertices, it is shown in \cite{PS} that the
monomial ideal $M_G$ has a minimal cellular resolution supported on the first
barycentric division of a $(n-1)$-simplex.  This fact is used in \cite{PS},
where the authors describe resolutions of the lattice ideal $I_G$ in the case
that $G$ is a \emph{saturated} graph.  Indeed in this case the monomial ideal
in question is generic, and the resolution coincides with the \emph{Scarf
complex} of $M_G$.  By results of \cite{PeeStu}, this resolution lifts to the
Scarf complex of the lattice ideal $I_G$.  In \cite{PS} the authors show that
the barycentric subdivision of an $(n-1)$ simplex supports a resolution of
$M_G$ for an arbitrary graph $G$ on $n$ vertices. However, these resolutions
are typically far from minimal and thus directly
accessible homological information (such as (graded) Betti numbers) is limited.
In both \cite{PS} and \cite{MS} the issue of finding a minimal resolution for
the case of a general graph $G$ is left as an open question.

In this paper we describe a simple and explicit minimal cellular
resolution of the monomial ideal $M_G$. The polyhedral complex supporting the
resolution is obtained from the graphical hyperplane arrangement $\A_G$
associated to $G$. More precisely, the intersection of $\A_G$ with a certain
affine subspace yields the essential affine hyperplane arrangement $\tA_G$.
Our main result (Theorem \ref{thm:main}) is that $\B_G$, the bounded subcomplex
of $\tA_G$, supports a minimal free resolution of the monomial ideal $M_G$.  As
a corollary of our result we verify a conjecture of Perkinson et al.\
regarding the Betti numbers of $M_G$.  In particular the Betti numbers of
$M_G$ enumerate acyclic orientations of certain contractions of the graph $G$
(see Corollary \ref{cor:Betti_nums}).

In their work on Laplacian lattice ideals, Manjunath and Sturmfels~\cite{MS}
demonstrate how the duality involved in the discrete Riemann-Roch for certain
graphs can be expressed in terms of $M_G^*$, the ideal Alexander dual to $M_G$
with respect to the {canonical
monomial} of $G$.   Our constructions also lead to
an explicit description of a (co)cellular resolution of the ideal $M_G^*$ for all graphs $G$; see
Section~\ref{sec:duality}.

This collaboration began in Berlin in the summer of 2012, and the results of
this paper were first presented in the Combinatorics seminar at the University
of Miami in September 2012. While this paper was being prepared, the two
preprints~\cite{MSW} and~\cite{MohSho} were posted on the arXiv announcing
similar results. Both papers employ purely algebraic/combinatorial methods
while our perspective is that of geometric combinatorics.  In recent
conversations with the authors of~\cite{MohSho} we were made aware of their
independent work-in-progress involving cellular resolutions.

\section{Graphs, $G$-parking functions, and monomial ideals}\label{sec:graphs}

Throughout the paper we let $G = (V,E)$ be a finite, undirected graph on the
vertex set $V = [n+1] = \{1, 2, \dots, n+1\}$ and edge set $E$. We assume that
$G$ is connected and without loops but with possibly parallel edges, i.e.,
multiple edges between vertices $i$ and $j$.

We let $\Lap(G)$ denote the \Defn{Laplacian} of $G$.  Recall that $\Lap(G)$ is the symmetric
$(n+1) \times (n+1)$ matrix with $\Lap(G)_{ij} = -|\{ {\text{edges
between $i$ and $j$}}\}|$ if $i \not= j$ and equal to the degree of $i=j$, otherwise.
We will denote by $\Laplat(G)$ the sublattice of $\Z^n$ generated by the
rows of $\Lap(G)$. The Laplacian has been studied in various combinatorial settings including spectral graph theory~\cite{godsil}. Since $G$ is assumed to be connected, $\Lap(G)$ has a
one-dimensional kernel spanned by the vector $(1,1, \dots, 1)^t$.  The celebrated Matrix-Tree theorem (see~\cite[Sect.~13.2]{godsil}) asserts that $|\det \tilde\Lap(G)|$ is the number
of spanning trees of the graph $G$. This is an application of the Binet-Cauchy
theorem to the \Defn{truncated Laplace matrix} $\tilde\Lap(G)$, the matrix
obtained by deleting the $(n+1)$st (or any other) row and column from
$\Lap(G)$.  More recently, the Laplacian of $G$ has
appeared in the context of a discrete Riemann-Roch theory for graphs \cite{BN}, where it
encodes the dynamics of the so-called chip-firing moves (the discrete analogue
of linear equivalence of divisors).

We fix a field $\k$ and let $\k[x_1,x_2, \dots, x_{n+1}]$ denote the polynomial
ring in $n+1$ generators.  Associated to our graph $G$ we let $I_G$ denote the
\emph{lattice ideal} associated to $\Laplat(G)$. By definition $I_G$ is generated
by binomials of the form ${\bf x}^{\bf u} - {\bf x}^{\bf v}$, where ${\bf u} -
{\bf v}$ is in the lattice $\Laplat(G)$, generated by the columns of $\Lap(G)$,
\begin{equation}\label{Defn_of_IG}
    I_G \ = \ \langle {\bf x}^{\bf u} - {\bf x}^{\bf v}: {\bf u}, {\bf v} \in
    \N^{n+1}, {\bf u} - {\bf v} \in \Laplat(G) \rangle.
\end{equation}
Here we use the notation ${\bf x}^{\bf u} = x_1^{u_1}x_2^{u_2} \cdots
x_{n+1}^{u_{n+1}}$.  The ideal $I_G$ is called the \Defn{toppling ideal} of
the graph $G$ in~\cite{MS} and~\cite{PPW}.

The binomial ideal $I_G$ has a distinguished monomial initial ideal $M_G$,
characterized by the property that the standard monomials of $M_G$ are given
by the \emph{$G$-parking functions}.  The ideals $M_G$ have also been studied
in the context of monotone monomial ideals in \cite{PS}, and can be described
explicitly as follows. For any nonempty subset $I \subseteq [n]$ we define the
monomial
\begin{equation}\label{def:m_I}
    m_I \ := \ \prod \bigl\{ x_i : ij \in E, i \in I, j \in [n+1]\setminus I\}
    \ = \ \prod_{i \in I} x_i^{d_I(i)},
\end{equation}
where $d_I(i)$ denotes the number of edges from the vertex $i$ to a vertex in
$[n+1]\setminus I$.  Now define $M_G$ to be the ideal in $R = \k[x_1, \dots, x_n]$
generated by all $m_I$, as $I$ ranges over all nonempty subsets of $[n]$:
\begin{equation}\label{def:M_G}
    M_G \ = \ \langle m_I: \emptyset \neq I \subseteq [n] \rangle \ \subseteq
    \ \k[x_1,\dots,x_n].
\end{equation}
Since $G$ is connected and $n+1 \not\in I$, $M_G$ is a proper monomial ideal
in $R$. For a subset $I \subseteq [n+1]$, let us denote by $G[I]$, the
\Defn{vertex-induced subgraph}, i.e.\ the graph with vertex set $I$ and edges
$\{ ij \in E: i,j \in I\}$.

\begin{prop}\label{prop:min_gen}
    The monomial $m_I$ is a minimal generator of $M_G$ if and only if
    $G[I]$ and $G[I^c]$ are connected.
\end{prop}
\begin{proof}
    Suppose $G[I] = G[I_1] \uplus G[I_2]$. Then $m_{I} = m_{I_1}\cdot m_{I_2}$.
    Similarly, if $G[I^c] = G[J_1] \uplus G[J_2]$ and $n+1 \in J_2$, then
    $m_{I \cup J_1}$ divides $m_I$. Thus, $G[I]$ and $G[I^c]$ are connected
    for every minimal generator.

    Conversely, assume that $I \subseteq [n]$ satisfies the condition and
    assume that $m_J$ divides $m_I$ with $m_J \neq m_I$. Observe that $J$ cannot be a subset of
    $I$. By the definition of $m_I$ given in (\ref{def:m_I}) above we see that
    \begin{equation}\label{eqn:subsets}
        J \ \subseteq \  I \qquad \Longrightarrow \qquad d_J(j) \ \ge \ d_I(j) \quad
        \text{ for all } j \in J.
    \end{equation}
    %and $d_J(j) > d_I(j)$ for some $j \in J$ if the inclusion $J \subset I$ is
    %strict and $G[I]$ and $G[J]$ are connected.

    Thus, let $j \in J \setminus I$. Along a path from $j \in I^c$ to $n+1$
    inside $G[I^c]$, there is an edge $st \in E$ such that $s \in J$ but $t
    \in J^c$. Since $s \in I^c$ we have $d_I(s) = 0$.  But then $d_J(s) > 0 = d_I(s)$ which contradicts that
    $d_J(i) \le d_I(i)$ for all $i \in [n]$.
\end{proof}

According to~\cite{BCT}, a \Defn{$G$-parking function} is a function $\phi :
[n] \rightarrow \Z_{\ge 0}$ such that for every $\emptyset \not= I \subseteq
[n]$, there is an $i \in I$ such that $0 \le \phi(i) < d_I(i)$. It is easily
seen that $\phi: [n] \rightarrow \Z_{\ge 0}$ is a $G$-parking function if and
only if $\x^\phi = x_1^{\phi(1)}x_2^{\phi(2)}\cdots x_n^{\phi(n)} \not\in
M_G$. Thus the $G$-parking functions of $G$ are exactly the non-vanishing monomials
in $\k[\x]/M_G$. It is known~\cite{Gab} that the number of $G$-parking
functions (and hence the $\k$-vector space dimension of $\k[\x]/M_G$) equals the
number of spanning trees of $G$ which, as we have seen, is given by
$|\det \tilde \Lap(G)|$.

To realize $M_G$ as an initial ideal of $I_G$, we first fix a spanning tree
$T$ of $G$ rooted at the vertex $n+1$, and choose an ordering of the variables
that is a linear extension of $T$. More concretely, we choose a total order
$\preceq$ of $\{x_1,\dots,x_{n+1}\}$ that satisfies $x_i \succ x_j$ if in $T$
the vertex $i$ lies on the unique path from $n+1$ to $j$.  We then take the
reverse lexicographic term order for this ordering of variables.  We call any
such monomial ordering a \Defn{spanning tree order}, borrowing the term from
\cite{MS}.

\begin{example}
    If $G = K_{n+1}$ is the complete graph  on $n+1$ vertices, we can take the
    usual reverse lexicographic term order. The ideal $M_G$ is the
    \Defn{tree ideal} on $n$ variables, as described in
    \cite[Sect.~4.3.4]{MilStu}. In this case the standard monomials of $M_G$
    are given by the classical parking functions studied in combinatorics,
    i.e., sequences of non-negative integers $(a_1, a_2, \dots, a_n)$ with the
    property that, when placed in weakly ascending order, sit below a
    staircase: $a_i \le |\{ j : a_j \le a_i \}|$ for all $i \in [n]$. The
    \emph{maximal} parking functions are given by the permutations of $n$,
    which also correspond to the generators of the ideal that is Alexander dual to
    $M_G$.
\end{example}

\begin{example}
    A graph $G$ is said to be \Defn{saturated} if $G$ has $u_{ij} > 0$ edges
    between any two vertices $i \neq j$.  In \cite{MS} it is shown that if $G$ is
    a saturated graph on $n+1$ vertices then the ideal $I_G$ is a generic
    lattice ideal. For such graphs an explicit set of $2^n - 1$ binomial
    generators is given that form a Gr\"obner basis of $I_G$ with respect to
    reverse lexicographic order.
\end{example}

We next introduce the example that will be used throughout the paper.

\begin{example}[Running example]
    Let $G$ be the 4-cycle with vertices $\{1,2,3,4\}$ and edges
    $\{12,23,34,14\}$.  The Laplacian $\Lap(G)$ is the $4 \times 4$ matrix
    given below.
    \[
        \Lap(G) \ = \ \left[
            \begin{array}{rrrr}
                2 & -1 & 0 & -1 \\
                -1 & 2 & -1 & 0 \\
                0 & -1 & 2 & -1 \\
                -1 & 0 & -1 & 2
            \end{array} \right].
    \]
    The ideal $I_G$ is given by
    \[
        I_G \ =  \ \langle x_1^2 - x_2x_4, x_2^2 - x_1x_3, x_3^2 - x_2x_4,
        x_4^2 - x_1x_3, x_1x_2 - x_3x_4, x_2x_3 - x_1x_4 \rangle.
    \]
    The monomial ideal $M_G$ is given by
    \[
        M_G = \langle x_1^2, x_2^2, x_3^2, x_1x_2, x_1^2x_3^2, x_2x_3, x_1x_3
        \rangle.
    \]
    The generator $m_{\{1,3\}} = x_1^2x_3^2$ is redundant, so $M_G = \langle
    x_1^2, x_2^2, x_3^2, x_1x_2, x_2x_3, x_1x_3 \rangle = \langle x_1, x_2,
    x_3 \rangle^2$.
\end{example}

\section{Cellular Resolutions}

Let $R = \k[x_1, \dots, x_n]$ denote the polynomial ring on $n$ variables with
the standard $\Z^n$-grading given by $\deg(\x^a) = a \in \Z^n_{\ge 0}$ for any monomial
$\x^a$.  For any graded $R$-module $M$, a $\Z^n$-graded \Defn{free resolution}
of $M$ is an exact sequence
\[
    0 \ \leftarrow \ M \ \xleftarrow{\phi_1} \ F_1 \ \xleftarrow{\phi_2}
    \ \cdots \ \xleftarrow{\phi_r} \ F_r \ \leftarrow \ 0,
\]
where each $F_i$ is a graded free $R$-module
\[
    F_i \ \cong \ \bigoplus_{\sigma \in \Z^n} R(-\sigma)^{\beta_{i,\sigma}}
\]
and where each $\phi_i$ is a graded map.
The resolution is called \Defn{minimal} if each of the $\beta_{i,\sigma}$ is
minimal among all graded free resolutions of $M$. In this case the
$\beta_{i,\sigma} = \beta_{i, \sigma}(M)$ are called the \Defn{finely} or
\Defn{$\Z^n$-graded Betti numbers} of the module $M$. The \emph{$\Z$-graded}
Betti numbers of $M$ are given by
\[
    \beta_{i,j}(M) \ = \ \sum_{|\sigma| = j} \beta_{i,\sigma}(M)
\]
where $|\sigma| = \sigma_1 + \sigma_2 + \cdots + \sigma_n$.
And finally, for each integer $i$, the $i$th (ungraded) Betti number of $M$ is given by
\[
    \beta_i(M) \ = \ \sum_j \beta_{i,j}(M).
\]

\newcommand\X{\mathcal{X}}

A \Defn{labeled polyhedral complex} is a polyhedral complex $\X$ together
with an assignment $a_F \in \N^{n}$ to each face $F \in \X$ such that
\[
    (a_F)_i \ = \ \max \{ (a_G)_i : G \subset F \}.
\]
for all $i = 1,2,\dots,n$.
Let $\X$ be a labeled polyhedral complex and let
\[
    M \ = \ M_\X \ = \ \langle\,\x^{a_F} : F \in \X\,\rangle
    \   \subseteq \   \k[x_1,\dots,x_n]
\]
be the monomial ideal generated by the labels. Clearly, $M_\X$ is generated by
the $0$-cells of $\X$. For $\sigma \in \N^n$, let $\X_{\le \sigma} \subseteq
\X$ be the subcomplex of faces $F$ for which $a_F \le \sigma$ componentwise.
Bayer and Sturmfels~\cite{BS} described how labeled complexes can encode
resolutions of $M$, that is, when the chain complex ${\mathcal F}_{\X}$ of
$\X$ gives rise to a $\Z^n$-graded resolution---a  \Defn{cellular resolution}---of $M$.  We refer to~\cite{MilStu} for further details from which the
following criterion is taken.

\begin{prop}[{\cite[Prop.~4.5]{MilStu}}]\label{prop:cell_res}
    Let $\X$ be a labeled polyhedral complex and let $M = M_\X \subset
    \k[x_1,\dots,x_n]$ be the associated monomial ideal.
    Then $\X$ supports a cellular resolution ${\mathcal F}_{\X}$ of $M$ if and only if the
    subcomplex $\X_{\le \sigma}$ is $\k$-acyclic for all $\sigma \in \N^n$.
    The resolution is minimal if $a_F \not= a_G$ for all faces $F \subset G
    \in \X$.
\end{prop}

Resolutions of $I_G$ and $M_G$ provide a new perspective on the duality
expressed in the Riemann-Roch theory of the graph $G$ via the notion of
Alexander duality.  The Betti numbers of $M_G$ have combinatorial
interpretations in terms of the graph $G$, and, as we will see, also relate to
certain well-studied geometric complexes.

In~\cite{MS} the authors consider resolutions of $I_G$ and $M_G$. They show
that in the case of saturated graphs $G$, the toppling ideal $I_G$ has a
minimal cellular resolution supported on what we will denote as
$\B(\Delta_{n-1})$, the first barycentric subdivision of an $(n-1)$-dimensional
simplex.

\begin{figure}[H]
\begin{center}
  \includegraphics[scale = .9]{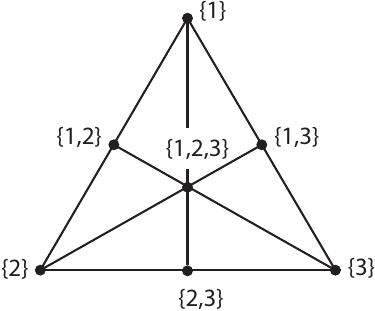}
\label{resolutioncomplete}
    \caption{The (minimal) free resolution of $M_G$ for $G = K_4$ is supported
    on $\B(\Delta_2)$.  Vertices are labeled by the subsets $I \subseteq [n]$
    corresponding to generators $m_I$.}

\end{center}
\end{figure}

In this case the complex $\B(\Delta_{n-1})$ can be lifted to the
\emph{Scarf complex} associated to the (in this case generic) lattice ideal
$I_G$. This extends a result from \cite{PS}, where it is shown that the
monomial ideal $M_G$ for $G=K_{n+1}$ has a resolution supported on the same
complex $\B(\Delta_{n-1})$.

For an arbitrary (not necessarily saturated) graph $G$, it is shown in
\cite{MS} that the complex $\B(\Delta_{n-1})$ supports a generally non-minimal
cellular resolution of $I_G$.  This leads to a formula for the Betti numbers
of $M_G$ in terms of the ranks of the reduced homology of certain induced
subcomplexes of $\B(\Delta_{n-1})$, although it is not clear how one might
obtain an explicit expression.

The description of a \emph{minimal} resolution of $M_G$ for an arbitrary graph
$G$ is stated as an open question in both \cite{MS} and \cite{PS}.  In
\cite{PPW} the authors provide a conjecture for the Betti numbers of $M_G$ in
the context of combinatorial data arising from Riemann-Roch theory of $G$.  In Section \ref{sec:cell_graph} we confirm this conjecture.

\section{Chip-firing and superstable
configurations}\label{sec:chip}

The Laplacian $\Lap(G)$ describes the dynamics of the so-called
\emph{chip-firing game} or \emph{Abelian Sandpile Model} for $G$.  We outline
the basic construction here and refer to~\cite{PPW} for further details.  In
this context, let us choose $n+1$ as a fixed sink.  By a \Defn{configuration} $c$
on $G$ we mean a placement of a number $c_i$ of `chips' or `grains of sand' on
each non-sink vertex $i \in [n]$. A vertex $i$ is said to be \Defn{unstable}
if the number of chips $c_i$ on $i$ is greater or equal to the degree of $i$.
In this case, one can then `fire' the vertex $i$, distributing chips to each
of its neighbors (one chip along every edge).  Chips that are sent to the sink
vertex disappear. It is easily seen that `firing' the vertex $i$ corresponds
to subtracting the $i$-th column of the truncated Laplacian $\tilde\Lap(G)$
from
$c$. If $G$ is connected, repeated chip-firing of
vertices eventually leads to a stable configuration $\bar c$.  One of the
fundamental results for chip-firing games is that this configuration
$\bar c$ is independent of the sequence of firings; we call $\bar c$ the
\Defn{stabilization} of $c$.

A stable configuration $c$ on $G$ is said to be \Defn{recurrent} if $c$ has a
nonnegative value on every non-sink vertex $[n]$ and if for any configuration
$a$ there exists a configuration $b$ such that $a+b$ stabilizes to $c$. The
stabilization process gives rise to a monoidal structure on the set of all
configurations, and forms a group when restricted to the recurrent
configurations of $G$.  This group is called the \emph{sandpile group} of $G$,
denoted $\mathcal{S}(G)$. The sandpile group can also be realized by
considering the discrete subgroup $\mathbb{L}_G \subset \Z^n$ generated
by the columns of the truncated Laplacian $\tilde\Lap(G)$.  We then have an
isomorphism
\[
    \mathcal{S}(G) \ \cong \ \Z^n/\mathbb{L}_G
\]
given by $c \mapsto c + \mathbb{L}_G$.  Thus
every element of $\Z^n$ is equivalent to a
recurrent configuration modulo the reduced Laplacian.  As a
corollary to the Matrix-Tree theorem we then see that the order of
$\mathcal{S}(G)$ is given by the number of spanning trees of $G$.

The \Defn{canonical configuration} $c_{\omega}$ on $G$ is
the maximally stable configuration given by
\[
(c_{\omega})_i \ = \ |\{ ij \in E : j \in [n]\}| - 1
\]
for every non-sink $i \in [n]$.  In the graph-theoretic
Riemann-Roch theory developed in \cite{BN}, configurations are naturally
identified with `divisors' on the graph $G$.  In this context the configuration $c_{\omega} - {\bf 1}$
is closely related to the `canonical divisor' of $G$.

To fully describe our results we will need a few more notions from the
chip-firing literature. As opposed to firing one vertex at a time, we consider a rule where one may fire sets of vertices simultaneously. This leads
to a stronger version of stability and a resulting notion of
\emph{superstable} configurations.  We will not detail the construction here
as it will be enough for us to use the following characterization (\cite{PPW}): a sequence
$a = (a_1, a_2, \dots, a_n)$ is a \Defn{superstable configuration} of $G$ if
and only if $a$ is a $G$-parking function. For the next result, recall that an \Defn{acyclic
orientation} of $G$ is an orientation $\mathcal{O}$ of the edges of $G$ with no directed cycles.

\begin{thm}[{\cite[Thm.~3.1]{BCT}}]
    There is a bijection between the set of acyclic orientations of $G$ with
    unique sink $n+1$ and the set of maximal superstable configurations of $G$.  Given an acyclic orientation $\mathcal{O}$, the
    corresponding configuration $c = c^{\mathcal{O}}$ is given by
    \[
        c_i \ = \ |\{ i \rightarrow_{\mathcal{O}} j : j \in [n+1]\}| - 1.
    \]
\end{thm}

We note that the bijection described in \cite{BCT} actually involves the set of acyclic orientations with unique \emph{source} $n+1$, but of course there is a bijection between these and orientations with unique sink $n+1$ by reversing all arrows.

In \cite{PPW} (Corollary 5.15) it is shown that a configuration $c$ is superstable if and only if $c_{\omega} - c$ is recurrent. Hence for an undirected graph $G$ on vertex set $[n+1]$, there is a bijective correspondence between: minimal recurrent configurations, maximal superstable
configurations, maximal $G$-parking functions, and acyclic orientations with
$n+1$ as the unique sink vertex.

\section{Graphical Arrangements and Whitney numbers}

In this section we introduce graphical hyperplane arrangements and other
notions from geometric combinatorics that will serve to describe our
resolutions.  We describe the basic constructions and terminology here, and
refer to~\cite{GreZas} and~\cite{StaHyp} for more details.

Once again we fix our graph $G$ on vertex set $V = [n+1]$ with edge set $E$.
For an edge $ij \in E$, we define a corresponding hyperplane
\[
    h_{ij} \ := \ \{x \in \R^{n+1}: x_i = x_j \}.
\]
The arrangement $\A_G = \{h_{ij} : ij \in E\}$ of hyperplanes in $\R^{n+1}$ is
called the \Defn{graphical arrangement} of $G$.  The arrangement $\A_G$
dissects space into polyhedra, called \Defn{cells}, of various dimensions, and
we use $f_k(\A_G)$ to denote the number of $k$-dimensional cells, or
$k$-cells, for short.  A \Defn{flat} of $\A_G$ is a nonempty intersection of
elements of ${\mathcal A}_G$, and we let $L_G = L(\A_G)$ denote the collection
of flats partially ordered by reverse inclusion.  In fact $L_G$ is a ranked
(geometric) lattice, called the \Defn{lattice of flats}, with minimum
$\R^{n+1}$ and maximum $\hat{1} = \bigcap_{ij \in E} h_{ij}$. The rank
$\mathrm{rk}_G(x)$ of $x \in L_G$ is given by $n+1 - \dim x$ and, since $G$ is
connected, $L_G$ has total rank $ \mathrm{rk}(\hat{1}) = n$ (which we also
take to be the rank of $\A_G)$.
The \Defn{lattice of partitions} of $G$ is the collection of unordered
partitions $V = V_1 \uplus V_2 \uplus \cdots \uplus V_m$ such that the
vertex-induced graph $G[V_k]$ is connected for all $1 \le k \le m$. The
partial order
is by coarsening: a partition $\{V_k\}_k$ is smaller than $\{U_l\}_l$ if every
$U_s$ is contained in some $V_t$. The lattice of partitions is naturally
isomorphic to $L_G$ by associating to $\{V_k\}_k$, the flat
\[
\bigcap \Bigl\{ h_{ij} : ij \in E, \{i,j\} \subseteq V_k \text{ for some } k
    \Bigr\}.
\]
We will freely use both perspectives on the elements of $L_G$.

Central to the study of hyperplane arrangements (and more general matroids) is
the notion of Whitney numbers.  Here the \Defn{doubly indexed Whitney numbers}
of the first kind are given by
\[
    w_{ij}(L_G) \ = \
    \sum \left\{ \mu_{L_G}(x,y) \;:\; x,y \in L_G, \mathrm{rk}(x) = i,
    \mathrm{rk}(y) = j\right\}
    %\sum_{x^i \in L_G} \sum_{x^j \in L_G} \mu(x^i, x^j),
\]
where $\mu_{L_G}$ is M\"obius function of $L_G$.  The \Defn{Whitney numbers}
of $L_G$ are the simply indexed versions
\[
    w_j(L_G) \ = \ w_{0j}(L_G).
\]

There is a well known connection (see \cite{GreZas}) between the Whitney
numbers of $L_G$ and the chromatic polynomial of $G$ given by
\[
    \chi(t) \ = \ \sum_{j=0}^{n} w_j(L_G)t^{n-j}.
\]
We say that a hyperplane $H$ is in \Defn{general position} with respect to the
arrangement $\A_G$ if $\dim (x \cap H) = \dim x - 1$ for all flats $x \in L_G$.
If $\A$ is a hyperplane arrangement and $U$ an affine subspace not parallel to
any hyperplane, then the \Defn{restriction} of $\A$ to $U$ is given by $\A|_U =
\{ H \cap U : H \in \A \}$.  We will need some further results from
\cite{GreZas} that we collect here for future reference.

\begin{thm}[{\cite[Thm.~3.2]{GreZas}}]\label{thm:whitney_cells}
    Let $\mathcal{A}$ be an arrangement of linear hyperplanes of rank $r$ and
    let $k>0$.  Let $H$ be a hyperplane general with respect to $\A$.  Then
    the induced arrangement $\A|_H$ has $|\mu(0,1)| = |w_r(L_G)|$ relatively
    bounded regions and $|w_{d-k,r}|$ relatively bounded $k-1$ cells.
\end{thm}

\begin{cor}[{\cite[Cor.~7.3]{GreZas}}]\label{cor:whitney_orients}
    Let $G$ be a graph with vertex set $[n+1]$.  The number of acyclic
    orientations of all contractions $G/S$ in which $S \in L_G$ has $k$
    components, such that the vertex $n+1$ is the only sink, equals
    $|w_{n+1-k,n}(L_G)|$.
\end{cor}

\section{Cellular resolutions from graphical arrangements}
\label{sec:cell_graph}

In this section we describe our minimal cellular resolution of $M_G$ and
derive some consequences.  As above we fix a connected graph $G$ on vertex set
$[n+1]$ and let $\A_G \subset \R^{n+1}$ denote the associated graphical arrangement. Let
$U \subset \R^{n+1}$ be the affine subspace
\[
    U \ := \ \Bigl\{ x \in \R^{n+1} : x_{n+1} = 0,\, x_1 + x_2 + \cdots + x_n =
    1\Bigr\} \ \cong \ \R^{n-1}
\]
We let $\tA_G \ = \ \bigl\{ \th_{ij} :=  h_{ij} \cap U : ij \in E \bigr\}$ be the
restriction of $\A_G$ to $U$.  Note that $\tA_G$ is an essential arrangement of $|E|$ affine
hyperplanes (two hyperplanes $\th_{ij}$ and $\th_{kl}$ coincide iff $ij$ and
$kl$ are parallel).

For a point $p \in U$ such that $p_i \not= p_j$ for all $ij \in E$, we obtain
an orientation on $G$ by orienting $i \rightarrow j$ if $p_i > p_j$. It is
easy to see that this orientation is in fact acyclic.  If $p$ takes the same
value on some edges, we get an acyclic orientation on a certain contraction of
$G$ as follows.  From $p$ we get a partition of the vertex set $[n+1] = V_1
\uplus V_2 \uplus \cdots \uplus V_s$ where for each $k$ we place $i,j \in V_k$
if there is a path $i = i_0i_1\dots i_t = j$ in $G$ such that $p_{i_{h-1}} =
p_{i_{h}}$ for all $1 \le h \le t$. In particular, the induced subgraphs
$G[V_i]$ are connected and we denote by $G/p$ the result of contracting each
$G[V_i]$ to a single vertex. The remaining edges satisfy $p_i \not= p_j$ and
thus we get an acyclic orientation on $G/p$.

The \Defn{bounded complex} $\B_G$ is the polyhedral complex of bounded cells
of  $\tA_G$ in $U$.  We let $|\B_G|$ denote the underlying pointset in $U$.
The next result says that we can determine points $p$ of $|\B_G|$ in terms of
the associated graph $G/p$.

\begin{figure}[H]
\begin{center}
    \includegraphics[scale=0.9]{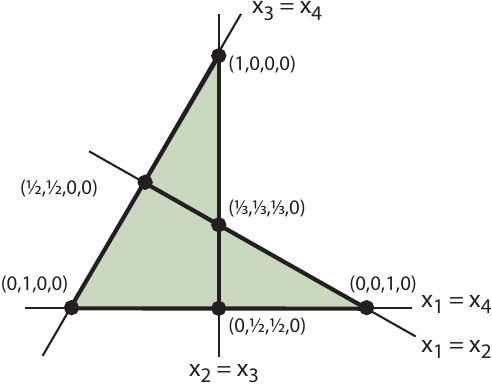}
\caption{The complex ${\mathcal B}_G$ for our running example.  The coordinates (as elements of $\R^4$) of the 0-cells are indicated.}
\label{fig:excomplex}
\end{center}
\end{figure}

\begin{prop} \label{prop:bnd_cell}
    For $p \in U$ let $C$ be the inclusion-minimal cell having $p$ in the
    relative interior. Then $\dim C + 2$ is the number of vertices of $G/p$, and
    $C \in \B_G$ if and only if the acyclic orientation on $G/p$ has a unique
    sink given by the vertex class containing $n+1$.
\end{prop}

\begin{proof}
    Let $L \subset U$ be the intersection of all $\th_{ij}$ for which $p_i =
    p_j$. Then $C$ is a full-dimensional cell in $L$ and we have to determine
    only $\dim L$. But the hyperplane arrangement induced in $L$ by $\tA_G$ is
    isomorphic to the arrangement $\tA_{G/p}$ and thus $\dim L = |V(G/p)| - 2$.
    It thus suffices to assume that $p_i \neq p_j$ for all edges $ij \in E$.

    If $G$ has more than one sink, let $i \neq n+1$ be one of them and let $j$
    be an arbitrary source. We claim that $p$ is not contained in a bounded
    cell. Consider the halfline $\{ p(t) = p + t(e_j - e_i) : t \ge 0\}$.
    Since $j$ is a source we have $p_j > p_k$ for all $jk \in E$ and $p(t)_j
    \ge p_j > p_k \ge p(t)_k$ for all $t \ge 0$. An analogous argument applies
    for $i$ and shows that the halfline is contained in the same
    inclusion-minimal cell of $\tA_G$ as $p$.

    Conversely, assume $C$ is unbounded. Since $C \subset U$, there is a point
    $q \in \relint(C)$ with $q_i < 0$ for some
    $i$. All points in the relative interior of $C$ induce the same
    orientation on $G$. However, since $q_{n+1} = 0 > q_i$ and $G$ is
    connected, there is no directed path from $i$ to $n+1$ and thus $n+1$ is
    not a sink.
%
%    Conversely, if $p$ is contained in an unbounded cell, then there is a
%    nonzero vector $u \in \R^{n+1}$ with $u_{n+1} = 0$ and $\sum_i u_i = 0$
%    such that $\{ p + t u : t \ge 0 \}$ does not meet any hyperplane
%    $\th_{ij}$ with $ij \in E$. If the corresponding orientation of $G$ has
%    only one sink $s \in [n+1]$, then, since $G$ is connected, $p_i > p_s$
%    for all $i \neq s$.  But, since $u \neq 0$, there is a $j \in [n]$ such
%    that $u_j < 0$.
\end{proof}

In order for $\B_G$ to support a cellular resolution, we have to label the
zero-dimensional cells of $\B_G$. From Proposition~\ref{prop:bnd_cell}, we
infer that a $0$-cell $v$ of $\B_G$ is of the form $v = \frac{1}{|I|}e_I$ where
$e_I \in \{0,1\}^{n+1}$ is the characteristic vector of the non-empty subset
$I \subseteq [n]$.  Using the fact that subsets $I \subseteq [n]$ correspond
to monomials $m_I$ according to \eqref{def:m_I}, this gives us a natural
labeling of $\B_G$.

\begin{cor}\label{cor:zerocells}
    Under the labeling described above, the $0$-cells of $\B_G$ are labeled by the minimal generators of
    $M_G$. The label $a_v \in \N^n$ of a $0$-cell $v = \frac{1}{|I|} e_I$ of
    $\B_G$ is given by
    \[
        (a_v)_i \ := \ d_I(i)
    \]
    for all $i \in [n]$.
\end{cor}
\begin{proof}
    From~\eqref{def:m_I} and \eqref{def:M_G}, we know that $\x^{a_v} = m_I$ is
    in $M_G$. Now, by Proposition~\ref{prop:min_gen}, $m_I$ is a minimal
    generator if and only if $G[I]$ and $G[I^c]$ are connected and, by
    Proposition~\ref{prop:bnd_cell} this is the case if and only if $v =
    \frac{1}{|I|}e_I$ is a $0$-cell of $\B_G$.
\end{proof}

\begin{figure}[H]
\begin{center}
    \includegraphics[scale=0.75]{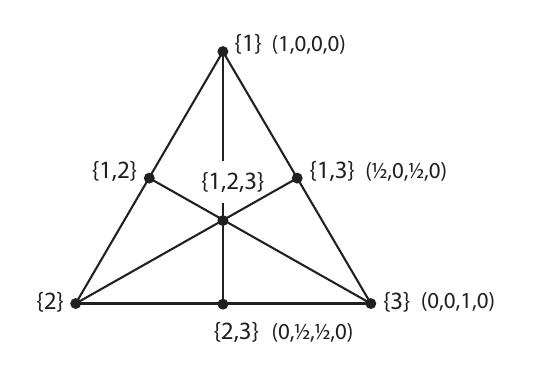}
    \includegraphics[scale=0.75]{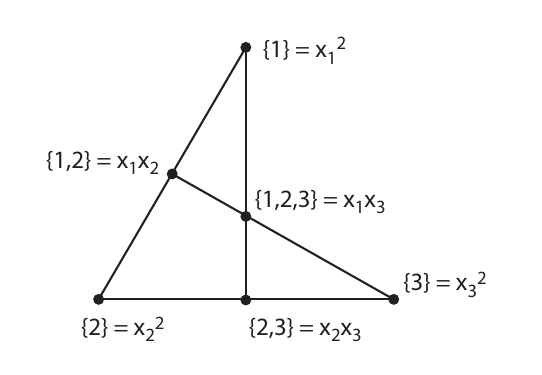}
\caption{The subset labeling of ${\mathcal B}_{K_{n+1}}$ corresponding to coordinates, and the induced monomial labeling of ${\mathcal B}_G$ for our running example $G$.}
\label{fig:monlabels}
\end{center}
\end{figure}

The labels on higher dimensional bounded cells are determined by the
componentwise maximum of the labels of incident $0$-cells.  That is, for a
cell $C \in \B_G$ and $i \in [n]$
\[
    (a_C)_i \ = \ \max\{ (a_v)_i : v \in C \textnormal{ is a $0$-cell} \}.
\]
However, we can also determine the labels of such cells directly.
\begin{prop}\label{prop:point_label}
    Let $C \in \B_G$ be a bounded cell with label $a_C \in \N^{n}$.
    For $p \in \relint C$, we have
    \[
    (a_C)_i \ = \ \#\{ij \in E : p_i > p_j \}
    \]
\end{prop}
\begin{proof}
    Fix $1 \le i \le n$ and let $D_i(p) := \#\{ij \in E : p_i > p_j \}$.
    Consider the set $I \subseteq [n]$ of vertices of $G$ such that $k \in
    I$ if there is a path $k = i_0i_1\dots i_s = i$ and $p_{i_j} \ge p_i$ for all
    $0
    \le j \le s$. By construction $G[I]$ is connected. Assume that the
    complementary graph $G[I^c]$ is disconnected and let $J \subset I^c$ be
    a connected component not containing $n+1$.  Since $C$ is a cell in the
    hyperplane arrangement $\tA_G$ and $p$ is a point in the relative
    interior, $C$ is the set of points $x \in U$ with
    \begin{align*}
        x_i = x_j \ \text{ if } p_i = p_j \quad  \text{ and} \quad
        x_i \ge x_j \ \text{ if } p_i > p_j
    \end{align*}
    for all $ij \in E$. Observe that for all $t \ge 0$ the point $p(t) = p +
    \tfrac{t}{|I|} e_I - \tfrac{t}{|J|} e_J$ is in $C$. Indeed, $p(t) \in U$
    for all $t$. As for the defining equations and inequalities of $C$, the
    only relevant case to check is $kl \in E$ with $k \in I$ and $l \in J$. By
    definition of $I$, we have $p_k > p_l$ and $p(t)_k = p_k + \frac{t}{|I|}
    > p_l - \frac{t}{|J|}$ for all $t \ge 0$. This shows that $C$ is not
    bounded.

    It follows that $G[I^c]$ is connected and, by
    Proposition~\ref{prop:bnd_cell}, the point $q = \tfrac{1}{|I|} e_I$ is a
    $0$-cell of $\B_G$. To see that $q \in C$ let $kl \in E$ such that $p_k >
    p_l$ but $q_k < q_l$. This implies $l \in I$ and $k \in I^c$ but since
    $p_k > p_l$ and there is a path from $l$ to $i$ with values $\ge p_i$,
    this means $k \in I$. As for monomial labels we have $(a_C)_i \ge (a_q)_i
    = \#\{ ij \in E: j \not\in I\} = D_i(p)$.

    To see that $D(p)_i$ upper bounds $(a_C)_i$, observe that $\relint(C)
    \subseteq \{ x \in U : x_i > x_j \text{ for all } ij \in E \text{ with }
    p_i > p_j \}$ for $i$ fixed. In particular $C$ is in the closure of this
    set and every point $q$ in the closure satisfies $D(q)_i \le D(p)_i =
    (a_C)_i$.
%    let $q \in C$ be any $0$-cell
%    To see that $D(p)_i$ upper bounds $(a_C)_i$, let $q \in C$ be any $0$-cell
%    and let $p(t) := (1-t)p + tq$. Since $C$ is a convex polytope defined by
%    the hyperplanes and halfspaces induced by $\tA_G$,
%
%    Since $C$ is convex, we have $D_i(p) \ge
%    D_i(p(t))$ for all $0 \le t \le 1$ and we have $D_i(p(1)) = (a_q)_i$.
\end{proof}

Here is the main lemma that we need.

\begin{lemma} \label{lem:main}
    Let $G = ([n+1],E)$ be a connected graph and $\B_G$ the bounded subcomplex labeled
    according to the monomial ideal $M_G$. For every $\sigma \in \N^n$, the
    set
    \[
    |B_G|_{\le \sigma} \ = \ \bigcup \{F \in B_G: a_F \le \sigma \}
    \]
    is star-convex.  Hence $|B_G|_{\le \sigma}$ is contractible and in particular $\k$-acyclic for any field $\k$.
    \end{lemma}
\begin{proof}
    Let $p^1, p^2, \dots, p^m$ be the $0$-dimensional cells of $(\B_G)_{\le
    \sigma}$ and for each $i$ let $a^i = a_{p^i} \in \N^n$ be the corresponding exponent vector of the monomial label. Let
    us define
    \[
        J \ = \ \bigcup_{i=1}^m \supp(a^i).
    \]
    In the subgraph of $G$ induced on the vertices $[n+1]\backslash J$, let
    $K$ be the set of vertices corresponding to the connected component containing the vertex $n+1$.  Finally, define
     \[
        I \ = \ [n+1] \backslash K \quad\text{ and }\quad q \ = \
        \frac{1}{|I|} e_I.
    \]
    We claim that $q$ is a star point of $|\B_G|_{\le \sigma}$. By
    construction, the contraction $G/q$ has a unique sink and thus $q$ is
    contained in the relative interior of some cell $\B_G$ with label $a_q$.

\begin{figure}[H]
\begin{center}
    \includegraphics[scale=0.9]{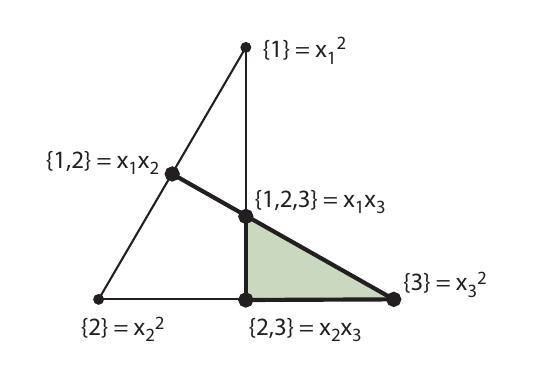}
\caption{The subspace $({\mathcal B}_G)_{\leq \sigma}$ for our running example $G$ with $\sigma = (1,1,2) = x_1x_2x_3^2$.  The star point in this case is $q = (\frac{1}{3}, \frac{1}{3}, \frac{1}{3}, 0)$.}
\label{fig:Cycledownset}
\end{center}
\end{figure}

    Let us next verify that $a_q \le \sigma$ so that indeed $q \in |\B_G|_{\le \sigma}$.  First we claim that $I_k := \supp(p^k)
    \subseteq I$ for all $k$ (note that $\supp(a^k) \subseteq \supp(p^k)$).   For this suppose that $i \in I_k$.  If $i \in \supp(a^k)$ then $i \in J$ and hence $i \in I$.  If $i \notin \supp(a^k)$ then by definition $d_{I_k}(i) = 0$.  Since $p_k$ is a $0$-cell of $\B_G$, both $G[I_k]$ and $G[I^c_k]$ are connected.  Therefore, any path from $n+1$ to $i$ has to include an edge $st$ with $s \notin \supp(a^k)$ and $t \in \supp(a^k)$.  Hence $i$ is not in the connected component of $n+1$ in the subgraph of $G$ induced on $[n+1]\setminus \supp(a^k)$.  Since $\supp(a^k) \subseteq J$ this in turn implies that $i \notin K$ so that $i \in I$. Now let $i \in I$ with $(a_q)_i = d_I(i) > 0$ and let $j \in I^c = K$ with $ij \in E$.
    Then there is a path from $n+1$ to $j$ that does not meet $J$.
    Hence, if $i \not\in J$, there is a path in $J^c$ from $i$ to $n+1$ which would contradict $i \in I$.
    This
    implies $i \in J$ and thus $i \in I_k$ for some $k$. Now observe that
    $\sigma_i \ge d_{I_k}(i) \ge d_{I}(i)$.

    Next, let $r \in |\B_G|_{\le \sigma}$ be an arbitrary point. We need to
    show that the line segment connecting $r$ and $q$ is contained in
    $|\B_G|_{\le \sigma}$.  Recall that if $i \in I$, then either $d_I(i) = 0$
    or else, by Proposition~\ref{prop:point_label}, there exists a point $p
    \in |\B_G|_{\le \sigma}$ such that $p_i > p_j$ for $ij \in E$.
    Thus, if $r_i > r_j$ for some $ij \in E$, we have $i \in I$ and $q_i =
    \frac{1}{|I|}$. But $q_s \le \frac{1}{|I|}$ for all $s \in [n]$ and
    therefore $q_i \ge q_j$. We conclude that no hyperplane $\th_{ij}$ strictly
    separates $r$ and $q$ and therefore the open line segment $(r,q)$ is
    contained in some (inclusion minimal) cell $C$ of the arrangement $\tA_G$.

    We first confirm that $C \in \B_G$. Let $p \in (r,q)
    \subseteq \relint C$.  Appealing to Proposition \ref{prop:bnd_cell} we need
    to show that the induced orientation on the contraction of $G/p$ has a
    unique sink given by the class containing the vertex $n+1$.  By
    contradiction, assume that $i \in [n]$ corresponds to a sink in $G/p$ that
    is different from $n+1$. Let $i = i_0i_1\dots i_m = n+1$ be a path such
    that $r_{i_{h-1}} \ge r_{i_{h}}$ for all $h = 1,\dots,m$. As $r$ is in the
    bounded subcomplex and hence $G / r$ has a unique sink, such a path
    exists. By assumption, the path is not weakly decreasing for $p$; that is, there
    is an index $l$ with $p_{i_{l-1}} < p_{i_{l}}$.
    In particular we have $p_{i_{l}} > 0$. We have $\supp(p) = \supp(r) \cup \supp(q)$ and, by construction, $\supp(r) \subseteq  \supp(q)$ and hence  $i_l \in I = \supp(q)$ and $q_{i_l}
    = \frac{1}{|I|}$. Thus, the path is weakly decreasing for $q$ which
    implies $p \in (r,q) \subseteq \{ x \in U : x_{i_{l-1}} \ge x_{i_l}\}$.  Recall that $U$ is the affine slice $U = \{ x \in \R^{n+1} : x_{n+1} = 0,\, x_1 + x_2 + \cdots + x_n = 1 \}$.

    It is left to show that $a_C \le \sigma$. For this let $i \in [n]$ with $p_i
    >0$, so that $i \in I$. If $r_i >0$ then since $\supp(r) \subseteq
    \supp(q)$, we have $q_i > q_j \Rightarrow r_i > r_j$ and thus $\sigma_i
    \ge (a_r)_i \ge (a_C)_i$. If $r_i = 0$, then $p_i > p_j \Rightarrow q_i >
    q_j$ and thus $\sigma_i \ge (a_q)_i \ge (a_C)_i$, as desired.
\end{proof}

\begin{thm}\label{thm:main}
With the monomial labeling described above, the complex $\B_G$ supports a minimal cellular resolution of $M_G \subset
\k[x_1,\dots,x_n]$ over every field $\k$.
\end{thm}

\begin{proof}
    We apply the criteria from Proposition \ref{prop:cell_res}.  Corollary~\ref{cor:zerocells} and Lemma~\ref{lem:main} imply that $\B_G$, with the monomial labeling described in Proposition \ref{prop:point_label}, supports a cellular resolution of $M_G$.  Proposition~\ref{prop:point_label} asserts that the cellular resolution is indeed minimal.
    \end{proof}

\begin{example}
In the case of our running example, we obtain a free minimal resolution of $M_G$ from the labeled complex in Figure \ref{fig:monlabels}.  Disregarding the grading, the resolution has the form
\[ 0 \ \leftarrow \ R \ \xleftarrow{\phi_1} \ R^6 \ \xleftarrow{\phi_2} \ R^8 \xleftarrow{\phi_3} \ R^3 \ \xleftarrow{\phi_4} \ 0.
\]
\noindent
In particular the Betti numbers $\beta_i$ are given by the face numbers $f_{i-1}(\B_G)$, and the differentials are described by the incidence relations of $\B_G$.
\end{example}

In addition we obtain an explicit combinatorial formula for the Betti numbers
of $I_G$, verifying a conjecture of Perkinson, Perlman, and Wilmes \cite[Conjecture 7.9]{PPW}.

\begin{cor} \label{cor:Betti_nums}
For a graph $G$ let $P_k$ denote the elements of $L_G$ that have $k$
components.  For $S \in P_k$ let $G/S$ be the graph induced on $S$, with $k$
vertices given by contracting the elements of $S$, while preserving the edges
between elements of $S$.  Then the (non-graded) Betti numbers of the ideal
$M_G$ are given by
\[
    \beta_k(M_G) \ = \ \sum_{S \in P_{k+1}} \#\{\textrm{$c$\;:\;$c$ a minimal
    recurrent configuration on $G/S$} \}.
\]
\end{cor}
\begin{proof}
Let $S$ be an element of $P_{k+1}$.  As we have seen, the minimal recurrent
configurations of $G/S$ correspond to the maximal $G/S$-parking functions,
which are turn in bijection with acyclic orientations of $G/S$ with unique
sink $n+1$.  From Proposition~\ref{prop:bnd_cell} we know that this set is
in bijection with the $k-1$ cells of ${\mathcal B}_G$.  Theorem \ref{thm:main} shows
that $k-1$ cells of ${\mathcal B}_G$ index $\beta_k(M_G)$.
\end{proof}

\begin{figure}[H]
\begin{center}
    \includegraphics[scale=0.87]{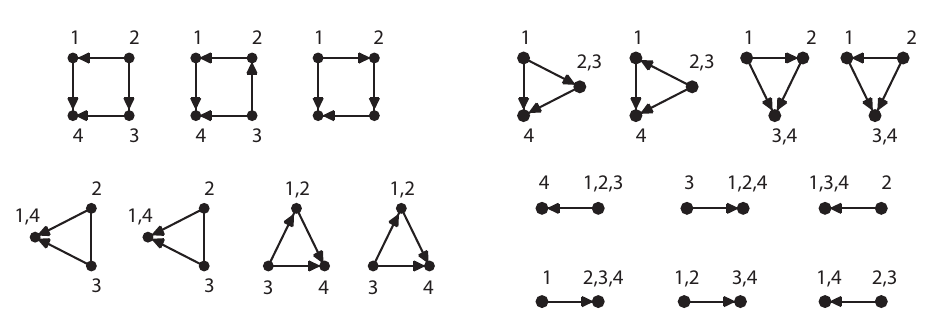}
\caption{The acyclic orientations of our example $G$ with vertex 4 as the unique sink, enumerating the syzygies of $M_G$: $\beta_3 = 3$, $\beta_2 = 8$, $\beta_1 = 6$.
}
\label{fig:orientations}
\end{center}
\end{figure}

\begin{cor}\label{cor:koszul}
    If $G$ is a tree on the vertices $[n+1]$, then $M_G = \langle x_1,
    x_2,\dots, x_n\rangle$ and $B_G$ is an $(n-1)$-dimensional simplex
    realizing the Koszul complex.
\end{cor}
\begin{proof}
    In this case $\tA_G$ is an arrangement of $n$ hyperplanes in $U \cong
    \R^{n-1}$. It is easy to see that the only acyclic orientation of $G$ with unique sink $n+1$ is
    obtained by orienting all edges towards $n+1$. And, since any contraction
    of $G$ is again a tree, we see that $\B_G$ has a unique bounded cell of
    dimension $n-1$ with $n$ facets, i.e. $\B_G$ is an $(n-1)$-dimensional simplex.
\end{proof}

Lastly we note that, in principle, Theorem~\ref{thm:main} gives an
\emph{algebraic} approach to acyclic orientations of $G$ and the study of the
face poset of $\B_G$ by means of minimal $\Z^n$-graded resolutions. However,
enumerating acyclic orientations of $G$ is $\#P$-hard~\cite{linial} and we do
not expect the algebraic method to be efficient.

\subsection{Duality}\label{sec:duality}

As detailed in \cite{MS}, the duality involved in the discrete Riemann Roch theory of certain graphs $G$ has an commutative algebraic analogue in terms of the ideal $M_G^* = M_G^{[{\bf k} + {\bf 1}]}$.  Here $M_G^*$ is the Alexander dual ideal of $M_G$ with respect to the monomial
\[
    \x^{\bf k + \bf 1} \ = \ x_1^{\deg(1)} x_2^{\deg(2)} \cdots x_n^{\deg(n)},
\]
where $\mathbf{1} = (1,1,\dots,1)$.

As a corollary to Theorem \ref{thm:main} we may apply the notion of duality of cellular resolutions \cite{MS} to obtain a (co)cellular minimal resolution of $M_G^*$.  For this we will need the following result from the literature.

\begin{prop}[{\cite[Thm~5.37]{MilStu}}]\label{prop:dual_res}
    Fix a monomial ideal $I$ generated in degrees preceding ${\bf a}$ and a
    cellular resolution ${\mathcal F}_{X}$ of $R/(I + \x^{\bf a + 1})$ such
    that all face labels on $X$ precede ${\bf a + 1}$.  If $Y = {\bf a} + {\bf
    1} - X$, then ${\mathcal F}^{Y_{\preceq {\bf a}}}$ is a weakly cocellular
    resolution of $I^{[{\bf a}]}$.  The resolution supported by $Y_{\preceq
    {\bf a}}$ is minimal if ${\mathcal F}_{X}$ is minimal.
\end{prop}

In our context we take $G$ to be our graph on vertex set $[n+1]$ and let $\B_G$ denote the labeled polyhedral complex defined above.  We define $\bar {\B}_G$ to be the \emph{colabeled} polyhedral complex with underlying complex $\B_G$ but with monomial label on a face $C$ given by
\[(\bar a_C)_i \ = \ \deg(i) + 1 - (a_C)_i.
\]

\noindent
Combining our Theorem \ref{thm:main} with Proposition \ref{prop:dual_res} gives us the following.

\begin{prop}
Let $G$ be a graph on vertex set $[n+1]$ and set
\[{\bf a} = (\deg(1), \deg(2), \dots, \deg(n)).\]
\noindent
Then with the notation established above, the labeled complex $\big(\bar{\B}_G \big)_{\preceq {\bf a}}$ supports a minimal cocellular resolution of the ideal $M_G^*$.
\end{prop}

\begin{example}
In our running example, we have ${\bf a} = (2,2,2)$, and
\[
M_G^* = M_G^{\bf a} = \langle x_1x_2^2x_3^2, x_1^2x_2x_3^2, x_1^2x_2^2x_3 \rangle.
\]
The colabeled complex $\bar {\B}_G$ is depicted below, along with the sub complex supporting the minimal resolution of $M_G^*$.
\end{example}

\begin{figure}[H]
\begin{center}
    \includegraphics[scale=0.85]{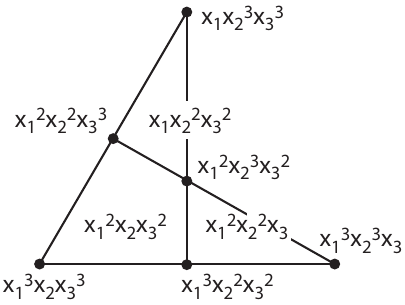} \hspace{.25 in}
    \includegraphics[scale=0.85]{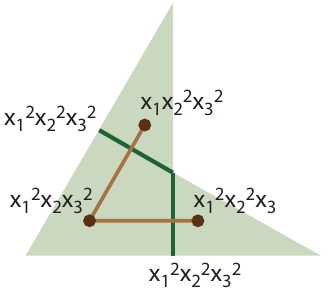}
\caption{The colabeled complex $\bar {\B}_G$ (with 0-cells and 2-cells labeled) and the subcomplex $\big(\bar{\B}_G \big)_{\preceq {\bf a}}$, consisting of three 2-cells and two 1-cells.  The dual complex is also depicted.}
\label{fig:dualres}
\end{center}
\end{figure}

\begin{rem}
For the case of $G = K_{n+1}$, the complete graph on $n+1$ vertices, this is the duality between the (resolutions of the) \emph{tree ideals} and the \emph{permutohedron ideals} described in {\cite[Example~5.44]{MilStu}}.
\end{rem}

\section{Further questions}

\subsection{Toppling ideals}

As mentioned in Section~\ref{sec:graphs}, there is a monomial term
order $\preceq$ for $\k[x_1,\dots,x_{n+1}]$ for which the ideal $M_G$ is the
initial ideal of $I_G$. A natural question to ask is whether one can describe
a minimal cellular resolution of the lattice binomial ideal $I_G$.  A
construction presents itself in terms of a quotient of the unimodular
\emph{graphic lattice} associated to $G$, similar to the cellular resolutions
of binomial Lawrence ideals described in \cite{BayPopStu}.  Mohammadi and
Shokrieh informed us about progress along these lines that will be published
in a sequel to their recent paper \cite{MohSho}.

\subsection{Monotone monomial ideals}

In \cite{PS} the authors study \emph{monotone monomial ideals}, a class of
monomial ideals that are strictly more general those arising as $M_G$ for a
graph $G$.  We recall the definition here.  A \Defn{monotone monomial family}
$\mathcal{M} = \{m_I:I \in \Sigma\}$ is a collection of monomials indexed by a
set $\Sigma$ of nonempty sets in $[n]$ that satisfy the conditions
\begin{itemize}
\item[(MM1)] For $I \in \Sigma$, $\supp(m_I) \subseteq I$,
\item[(MM2)] For $I, J \in \Sigma$ such that $I\subset J$, we have $m_I$
    divides $m_J$.
    %For $I, J \in \Sigma$ such that $I \subset J$ and $i \in I$, we have $\deg_{x_i}(m_I) \geq \deg_{x_i}(m_J)$.
    \item[(MM3)] For $I, J \in \Sigma$, $\mathsf{lcm}(m_I,m_J)$ is divisible by $m_k$ for some $K \supset I \cup J$ in $\Sigma$.
\end{itemize}

We then define the \Defn{monotone monomial ideal} $\langle \mathcal M \rangle$ associated to the family ${\mathcal M}$ to be the ideal generated by the monomials $m_I$ in ${\mathcal M}$.

The monomial ideals $M_G$ associated to a (directed) graph $G$ described above are monotone.  In this case $\Sigma$ is the set of all nonempty subsets of $[n]$ and $m_I$ is given by the formula \ref{def:m_I}.  The natural question to ask is if a similar construction to ${\mathcal B}_G$ can be used to resolve the ideals $\langle \mathcal M \rangle$.  Is there an arrangement of hyperplanes corresponding to a monotone family?

\subsection{Shi arrangements and duality}

In his study of the affine Weyl group of type $A_{n-1}$, Shi
introduced the arrangement ${\mathcal S}_n$ of hyperplanes in
$\R^n$ that now bears his name:
\[
    \mathcal{S}_n \ = \ \{x_i - x_j = 0,1 : 1 \leq i < j \leq n \}.
\]

Shi proved that the number of regions in the complement of
${\mathcal S}_n$ is given by $(n+1)^{n-1}$ (the number of trees on
$n+1$ labeled vertices).  Pak and Stanley \cite{Sta98} gave the
first bijective proof of this fact by providing an explicit
labeling of the regions with parking functions.  In \cite{HopPerBi}
Hopkins and Perkinson generalize this picture, motivated by a
conjecture of Duval, Klivans, and Martin.  Associated to a graph $G$
they define what they call a \emph{bigraphical arrangement}
and show that a Pak-Stanley type labeling of its regions are in
bijection with the $G$-parking functions.  Specifying certain parameters of the bigraphical arrangements recover the $G$-Shi and $G$-semiorder arrangements.

From Theorem \ref{thm:main}, we obtain a minimal resolution of
the ideal $M_G$ from the graphical arrangement of $G$.  We also
know that when $M_G$ is \emph{Riemann-Roch} (in the sense of \cite{MS}) the generators of $M_G^*$ (the Alexander dual of $M_G$) are
given by the maximal $G$-parking functions.  It would be interesting
to find a connection between the bigraphical arrangements of
\cite{HopPerBi} and the the cellular resolutions that we have
considered here.

\subsection{Topology of the partition poset}
In \cite{BjoWac} Bj\"orner and Wachs use the graphical hyperplane arrangement of the complete graph (the so-called \emph{braid arrangement}) to give an explicit basis for the homology of the partition poset.  It would be interesting to connect this study to the resolutions of the ideals studied here, and in particular to consider the case of a general graph $G$.

{\bf Acknowledgements.}
We would like to thank Michelle Wachs, Christian
Haase, Bernd Sturmfels, Katharina Jochemko, and Francesco Grande for stimulating
discussions.   We thank the two anonymous referees for very helpful comments and corrections.

\bibliographystyle{siam}

\bibliography{Laplacian_ideal_resolution}

\end{document}